\documentclass[12pt]{amsart}
\usepackage{amssymb} 
\usepackage{amsthm}
\usepackage{amscd}
\usepackage{multicol}
\usepackage[dvips]{color}
\usepackage[all]{xy}
\usepackage[dvipdfmx]{hyperref} 
\hypersetup{
    colorlinks=true,   
    linkcolor=blue,    
    citecolor=blue,    
    pdfborder={0 0 0}, 
    bookmarksopen=true 
}

\newtheorem{thm}{Theorem}[section]

\newtheorem{lem-dfn}[thm]{Lemma-Definition}
\newtheorem{prop}[thm]{Proposition}
\newtheorem{cor}[thm]{Corollary}

\newtheorem{conj}[thm]{Conjecture}

\theoremstyle{definition}
\newtheorem{defn}[thm]{Definition}

\newtheorem{ex}[thm]{Example}

\newtheorem{quest}[thm]{Question}

\newtheorem*{acknowledgement}{Acknowledgement}
\newtheorem{noname}[thm]{}
\theoremstyle{remark}
\newtheorem{rem}[thm]{Remark}

\numberwithin{equation}{section}

\DeclareMathOperator{\Spec}{Spec}

\DeclareMathOperator{\Proj}{Proj}

\DeclareMathOperator{\Hom}{Hom}

\DeclareMathOperator{\tr}{Tr}
\DeclareMathOperator{\Tr}{Tr}

\DeclareMathOperator{\di}{div}
\DeclareMathOperator{\Div}{Div}

\DeclareMathOperator{\chr}{char}

\DeclareMathOperator{\FRC}{frac}
\DeclareMathOperator{\Cl}{Cl}
\DeclareMathOperator{\cl}{cl}

\newcommand{\m}{\mathfrak m}

\newcommand{\KiR}{{K_R}^{-1}}

\newcommand{\PP}{\mathbb P}

\newcommand{\bbE}{\mathbb E}
\newcommand{\Z}{\mathbb Z}
\newcommand{\Q}{\mathbb Q}

\newcommand{\bbZ}{\ensuremath{\mathbb Z}}
\newcommand{\bbQ}{\ensuremath{\mathbb Q}}

\newcommand{\cA}{\mathcal A}

\newcommand{\cO}{\mathcal O}

\begin{document}
\title{Nearly Gorenstein normal graded rings}

\author{Tomohiro Okuma}
\address[Tomohiro Okuma]{Department of Mathematical Sciences, 
Yamagata University,  Yamagata, 990-8560, Japan.}
\email{okuma@sci.kj.yamagata-u.ac.jp}
\author{Kei-ichi Watanabe}
\address[Kei-ichi Watanabe]{Department of Mathematics, College of Humanities and Sciences, 
Nihon University, Setagaya-ku, Tokyo, 156-8550, Japan and 
Organization for the Strategic Coordination of Research and Intellectual Properties, Meiji University
}
\email{watnbkei@gmail.com}
\author{Ken-ichi Yoshida}
\address[Ken-ichi Yoshida]{Department of Mathematics, 
College of Humanities and Sciences, 
Nihon University, Setagaya-ku, Tokyo, 156-8550, Japan}
\email{yoshida.kennichi@nihon-u.ac.jp}
\dedicatory{Dedicated to the memory of J\"urgen Herzog}
\date{\today}
\thanks{TO was partially supported by JSPS Grant-in-Aid 
for Scientific Research (C) Grant Number 21K03215.
KW  was partially supported by JSPS Grant-in-Aid 
for Scientific Research (C) Grant Number 23K03040.
KY was partially supported by JSPS Grant-in-Aid 
for Scientific Research (C) Grant Number 24K06678.}
\keywords{Normal graded rings, nearly Gorenstein ring, rational singularity, 
trace ideal, elliptic singularity}
\subjclass[2020]{Primary: 13A02; Secondary: 14J17, 14B05, 13H10,13G05 13G05}

\begin{abstract} 
We investigate nearly Gorenstein property for a {\em normal} graded 
ring $R = \bigoplus_{n\ge 0}R_n$  finitely generated over a field. 
For that purpose, we investigate $\KiR$, the inverse of $K_R$ 
(the canonical module of $R$)  
 and introduce a new invariant $b(R)$ of $R$. 
We investigate nearly Gorenstein property of $R$ using $a(R)$ and 
$b(R)$ and $m(R)$, the initial degree of $R$. 
If $b(R)<0$, (and if $R$ is $\Q$-Gorenstein),  
then we believe that $R$ is log-terminal -- this is proved 
if $\dim R=2$ or $R$ is F-pure (or $F$-pure type).
\par
Then we determine the condition for a $2$-dimensional cone 
singularity over a smooth curve of genus $g\le 3$ to be nearly Gorenstein.   
We observe that \lq\lq almost Gorenstein" property and nearly 
Gorenstein property are drastically different for such rings.  
\end{abstract}

\maketitle

\section{Introduction}
\label{s:Intro}

For a Cohen-Macaulay local ring $(A,\m)$, Herzog-Hibi-Stamate 
defined \lq\lq {\em nearly Gorenstein} rings;   

\begin{defn}[{\cite[2.2]{HHS}}]
A Cohen-Macaulay local ring $(A,\m)$ with canonical module $K_A$ is 
said to be {\em nearly Gorenstein} 
if $\tr_A K_A \supset \m$. Here $\Tr_A K_A$ is the ideal 
 of $A$ generated by 
$\{ f(x)\;|\; f\in \Hom_A(K_A,A) \;{\rm and}\; x\in K_A\}$.
A Gorenstein ring $A$ is nearly Gorenstein by definition 
since $\Tr_A K_A =A$.  
\end{defn}

\par 
Note that if we consider $K_A$ as a fractional ideal in $Q(A)$, 
the total quotient ring of $A$, 
then we can identify $\Hom_A(K_A,A)$ with 
$K_A^{-1}:= \{ x\in Q(A)\;|\; xK_A\subset A\}$ 
and we can see $\Tr_A K_A = K_A\cdot K_A^{-1}$.   

\par 
In this paper, we investigate nearly Gorenstein property for 
 {\em normal graded domains}
finitely generated over a field $k=R_0$ 
making full use of gradings and also geometric properties
of the normal projective variety $\Proj(R)$. 
\bigskip
\par
Let us explain the structure of this paper. 
Throughout this paper, let 
$R = \bigoplus_{n\ge  0} R_n$ be a 
{\em normal graded domain} finitely generated over $R_0=k$,  a field. 
In the definition of nearly Gorenstein rings, 
the properties    \lq\lq Cohen-Macaulay" and 
\lq\lq $\Tr_R(K_R)=\m_R$" are completely independent.  
So, in \S 2 and \S 3, we investigate the property $\Tr_R(K_R)=\m_R$.

\par 
In \S 2, we explain the graded module structure of $K_R$, the canonical module of $R$ and 
$\KiR$, the inverse of  $K_R$ and introduce a new invariant 
\[
b(R).
\]

\par 
We explain necessary conditions for $R$ to be nearly Gorenstein using $a(R)$ (Goto-Watanabe's 
$a$-invariant), $b(R)$ and 
\[
m(R) = \min\{ n>0\;|\; R_n \ne 0\}.
\] 

\par 
Also, we discuss the property of $R$ if $b(R)<0$. This is analogous to the relation for $R$ 
to be rational singularity and $a(R)<0$. 

\par 
In \S 3, we discuss about Demazure's construction of normal graded rings 
which allows us more explicit discussion for normal graded rings. 

\par 
From \S 4, we discuss about the case $\dim R=2$, in which we have many nice results and 
bunch of examples. 
\par 
In \S 5, we investigate \lq\lq cone singularities", the case $D$ is an {\em integral} divisor
and determine all the nearly Gorenstein cone singularities $R(C,D)$ when the genus of $C$ is   
$2$ or $3$. 

\par
In \S 6, we recall the \lq\lq almost Gorenstein" property of $R$ 
and we compare nearly Gorenstein property and almost Gorenstein property 
for rings treated in \S 5.  
We see, surprisingly,  that two properties are {\em almost contradictory} for such rings.

\bigskip
\section{Normal Graded nearly Gorenstein Rings}
\label{s:Graded}

Throughout this paper, let $R = \bigoplus_{n\ge 0} R_n$ be a 
{\em normal  graded domain finitely generated} over $k= R_0$, a field. 
Then $K_R$ and $\KiR$ have natural structure of graded $R$-modules. 
It is well-known (cf. \cite{GW}) that $K_R$ is defined by 
\[
K_R= [H^d_{\m}(R)]^*, 
\]
where for a graded $R$-module $M =\bigoplus_{n\in \Z} M_n$, we define 
$M^*$ to be the $k$-dual of $M$ as graded $R$-modules. 
Namely, we define $M^{*}=\bigoplus_{n\in \Z} \Hom_k(M_n, k)$ 
 and the grading of $M^*$ is defined by $[M^*]_n = \Hom_k(M_{-n}, k)$.

\begin{defn}[\textrm{\cite{GW}}]
$a(R) = -  \min\{n\in \bbZ\;|\; [K_R]_n \ne 0\}$.  
\end{defn}
\par 
Now, let $Q(R) = S^{-1}R$, $S = \bigcup_{n\ge 0} R_n\setminus \{0\}$  
be the {\em homogeneous} total quotient ring of $R$. 
Since $K_R$ is a reflexive $R$-module of rank $1$, 
we can fix an isomorphism 
\[ 
K_R \otimes_R Q(R) \cong Q(R)
\] 
and by this identification, we put 
$\KiR = \bigoplus_{n\in \Z} [\KiR]_n$,  where
\[
[\KiR]_n = \{ x\in Q(R)_n \;|\; x K_R \subset R\}.
\]  
Note that in the divisor class group $\Cl(R)$, $\cl(\KiR) = - \cl(K_R)$. 
\par 
We also define 
\[
b(R) =  \min\{n\in \bbZ\;|\; [\KiR]_n \ne 0\} 
\]
and also $m(R) = \min\{n>0\;|\; R_n\ne 0\}$. 
Then we have an easy necessary condition for $R$ to be nearly Gorenstein. 

\begin{prop}\label{a(R)b(R)} 
Let $R$ be a normal graded ring as above. Then$;$  
\begin{enumerate}
\item $-a(R)+b(R) \ge 0$ and \lq\lq $=$" holds if and only if 
$K_R$ is a {\em free} $R$-module. 
If $K_R$ is not a free $R$-module, then 
we have $-a(R) + b(R) \ge m(R)$. 
\item If $\Tr_R(K_R)=\m_R$ holds, then we have $-a(R) + b(R) = m(R)$ and 
also we have $[K_R]_{-a(R)}\cdot [\KiR]_{b(R)} =  R_{m(R)}$. 
\item Assume that $R$ is standard graded $(R = k[R_1])$. 
If $-a(R)+b(R)=1$ and 
$[K_R]_{-a(R)}\cdot [\KiR]_{b(R)} =  R_1$, then we have 
$\Tr_R(K_R)= \m_R$. 
\end{enumerate}
\end{prop}

\begin{proof}  
The proof easily follows from the fact that for any $x\ne 0, x \in [K_R]_s$
(resp. $y\in [\KiR]_t$), the multiplication $x : [\KiR]_n \to R_{n+s}$ 
(resp. $y : [K_R]_n \to R_{n+t}$) is injective.  
\end{proof}

We know that if $R$ is a rational singularity, then $a(R)< 0$ (\cite{Fl}, \cite{W1}). 
We have a similar but stronger conjecture for $b(R)$ as follows.

\begin{conj}\label{b(R)<0} If $\chr(k) = 0$,  
$b(R)< 0$ and if $R$ has an isolated singularity, 
then $R$ is of strongly F-regular type.  
If, moreover,  $R$ is $\Q$-Gorenstein (that is, $r\; cl(K_R) = 0$ in $\Cl(R)$ 
for some positive integer $r$), then $R$ is log-terminal.
\end{conj}

\par 
This Conjecture is true if $\dim R=2$ and we can prove that 
in the case $p>0$, 
if $R$ is F-pure and $b(R)<0$, then $R$ is F-regular. Note that if 
$b(R)<0$, then by Proposition \ref{a(R)b(R)}, $a(R)<0$ 
and hence $R$ is a rational singularity 
if $\chr(k)=0$ and $R$ has an isolated singularity (cf. \cite{Fl}, \cite{W1}).  

\par 
We will discuss about this  in the next section. 

\section{Demazure Construction of Normal Graded Rings}  
\label{s:Demazure}

\begin{noname}\label{3.1}
Let $R$ be a normal Cohen-Macaulay graded ring as above. Then $X=\Proj(R)$ is a normal projective 
variety over $k$ with $\dim X=d-1$ if $\dim R=d$. Then by Demazure \cite{Dem},  
there is an ample $\Q$-divisor $D\in \Div(X)\otimes_{\bbZ}\bbQ$, such that 
\[R = \bigoplus_{n\ge 0} H^0( X,\cO_X( nD))T^n,\] 
where 
\[
H^0( X, \cO_X(nD)) = \{0 \ne f\in k(X)\;|\; \di_X(f) + nD \ge 0\} 
\cup \{0\}. 
\]  
(In this case, $D\in \Div(X)$ if and only if $R^{(n)}=k[R_n]$ for every $n\gg 1$; \cite{GW}). 
Note that we have always $H^0( X, \cO_X(nD)) = H^0( X, \cO_X([nD]))$, where $[nD]$ is the 
unique maximal integral divisor $\le nD$. 
\par
We can write 
\[
D= B - \sum_{i=1}^r (p_i/q_i) V_i,
\] 
where $B\in \Div(X), V_1,\ldots, V_r$ are different reduced and 
irreducible subvarieties of 
codimension $1$,  $p_i, q_i$ are relatively prime integers with  
$0<p_i<q_i$. 
Then we put  
\[
\FRC(D) = \sum_{i=1}^r \dfrac{q_i-1}{q_i} V_i. 
\]
\par 
Since  $H^{d-1}(X, \cO_X([-nD]))^{*} \cong H^0(X, \cO_X(K_X -[-nD]))$
by Serre duality, and   
$- [-nD] = [ nD + \FRC(D)]$,  we have;  
\[K_R = \bigoplus_{n\in \bbZ} H^0(X, \cO_X(K_X + \FRC(D)+ nD)) T^n;{\rm and}\] 
\[\KiR = \bigoplus_{n\in \bbZ} H^0(X, \cO_X(- K_X - \FRC(D)+ nD)) T^n. \] 
\par 
It is very convenient to consider $\deg(D)\in \Q$ for $D\in \Div(X)\otimes_{\Z}\Q$. 
\end{noname}

\begin{defn}\label{degD} 
We call an additive function $\deg : \Div(X) \to \Z$ a {\em degree} 
which induces $\deg :  \Div(X)\otimes_\Z \Q \to \Q$ which satisfies 
the conditions; 
\begin{enumerate}
\item[(1)] If $H^0(X,\cO_X(D))\ne 0$, then $\deg(D)\ge 0$, 
\item[(2)] If $\deg(D)=0$ and $H^0(X,\cO_X(D))\ne 0$, then  
$\cO_X(D)\cong \cO_X$.  
\end{enumerate}
\end{defn}
\par 
We can apply degree to nearly Gorenstein property. In the following Theorem \ref{degD-div}, 
\lq\lq deg" is any degree satisfying the conditions (1), (2) of Definition \ref{degD}. 

\begin{thm}\label{degD-div} Assume $R=R(X,D)$ satisfies the condition 
\lq\lq \; $\deg (K_X+\FRC(D)) =r \deg D$" for some positive 
integer $r$ and  $R_1\ne 0$. If  $K_R$ is {\bf not} a free $R$-module, 
then $R_1 \not\subset \Tr_R(K_R)$. 
Hence $\Tr_R (K_R)\subsetneq \m_R$ 
and $R$ is not nearly Gorenstein. 
\end{thm} 

\begin{proof} We will show that 
$-a(R) + b(R)\ge 2 > m(R)=1$. 
\par
Now, since $H^0(X, \cO_X( K_X+ \FRC(D) - a(R)D))\ne 0$, and 
$rD$ and $K_X+ \FRC(D)$ are not linearly equivalent,
we must have  $a(R) < r$. 
 On the other hand, 
since $H^0(X, \cO_X( -K_X- \FRC(D) + b(R)D))\ne 0$,
we must have $b(R) > r$. Thus we have 
$-a(R) + b(R) \ge (-r +1)+ (r+1) \ge 2$.   
\end{proof} 

\par 
Now we will discuss about Conjecture \ref{b(R)<0}. 
First we treat the case $\dim R=2$. 

\begin{prop} Assume $R=R(X,D)$ is as in $\ref{3.1}$. 
If $\dim R=2$ and $b(R)<0$, 
then $R$ is F-regular $($resp. log-terminal$)$ 
if $\chr(k)=p > 5$ $($resp. $\chr(k)=0)$. 
\end{prop}
\begin{proof} 
First note that if $b(R)<0$, then $R$ is a rational singularity 
and hence 
$\Cl(R)$ is a finite group (\cite{Li}). 
By definition of $b(R)$, we have $H^0(X, \cO_X(-K_X - \FRC(D) +b(R)D)) \ne 0$. 
This implies that $-K_X -\FRC(D) \ge - b(R)D$ 
and since $D$ is ample,  
we must have $\deg( -K_X) > 0$ and $X = \PP^1$. 
Also, if $\FRC(D) =\sum_{i=1}^r (q_i-1)/q_i V_i$, 
then from $\deg(-K_X -\FRC(D))>0$, we get $r\le 3$ and if $r=3$, then $(q_1, q_2, q_3)$ 
with $(q_1\le q_2\le q_3)$ is one of $(2,2,n), (2,3,3), (2,3,4), (2,3,5)$. 
\par \vspace{1mm}
Now, the injective envelope of $E_R(R/\m)$ is written as  
\[
{\bf E}= E_R(R/\m) = H^2_{\m}(K_R) 
= \bigoplus_{n\in \Z}H^1(X, \cO_X(K_X + \FRC(D) + nD))T^n
\]
and the socle $0:_{\bf{E}} \m = H^1(X, \cO_X(K_X + \FRC(D)))$.
Now the Frobenius map $F$ on the socle is given by 
\[
F : H^1(X, \cO_X(K_X + \FRC(D))) \to H^1(X, \cO_X(p[K_X + \FRC(D)])).
\]
Then we can show that $R$ is $F$-regular if $p>2$ (resp. $p>3$; resp. $p>5$) 
when $(q_1,q_2,q_3)=(2,2,n)$ (resp. $(2,3,3), (2,3,4)$; resp. $(2,3,5)$).
For example, if  $(q_1,q_2,q_3)=(2,3,5)$ and $p=5$, then  
$\deg([5(K_X + \FRC(D))])=-1$ and $H^1(X, \cO_X(5[K_X + \FRC(D)])=0$.  
\par 
This shows that $R$ is F-pure if $p>5$. 
Thus the $F$-regularity of $R$ follows from Theorem \ref{F-pureB} below. 
\end{proof}

\begin{thm} \label{F-pureB}
Assume that $p>0, b(R)<0$ and $R$ is $F$-pure. Then $R$ is F-regular. 
\end{thm}

The proof proceeds similarly as in \cite{Wat91}.
 
\begin{proof}[Proof of Theorem $\ref{F-pureB}$]
Since $H^0(X, \cO_X(-K_X -\FRC(D) + b(R)D ))\ne 0$, we have 
\[ -K_X - \FRC(D) = - b(R) D + B, \] 
where $B\ge 0$ and $- b(R) D$ is an ample $\Q$ divisor since $b(R)<0$. 
\par
Now, assume that $\chr(k)=p>0$. $R$ is strongly F-regular if for any 
$c\in R, c\ne 0$ (we may assume that $c$ is homogeneous), there exists $q= p^e$ such that there exists an $R$-homomorphism 
$\phi : R^{1/q} \to R$ with $\phi(c^{1/q})=1$. \par
Now, let ${\bf E} = E_R(R/\m)$ be the injective envelope 
of $R/\m$ and $\xi$ be a generator 
of the socle $[0:_{\bf E} \m]$ of $\bf E$. 
Then we see that the existence of $\phi$ such that 
$\phi(c^{1/q})=1$ is equivalent to say that $c^{1/q}(\xi\otimes 1) \ne 0$ in 
${\bf E}\otimes_R R^{1/q}$. Since $R$ is $F$-pure,   $(\xi\otimes 1) \ne 0$ in 
${\bf E}\otimes_R R^{1/q}$.
\par  
Now, since the injective envelope 
\[
{\bf E} \cong H^d_{\m}(K_R) =\bigoplus_{n\in \Z} H^{d-1}(X, \cO_X( K_X+\FRC(D)+nD)) T^n,
\]
where $\xi$ is given by a nonzero element of 
\[
H^{d-1}(X, \cO_X( K_X+\FRC(D)))= 
H^{d-1}(X, \cO_X( K_X))\cong k.
\]  
Then we have 
\[{\bf E}\otimes_R R^{1/q} \cong \left[
\bigoplus_{n\in \Z} H^{d-1}(X, \cO_X( q(K_X+\FRC(D))+nD)) T^n\right]^{1/q}\]
and taking dual, the right hand side is dual to  
\[
\left[ \bigoplus_{n\in \Z} H^0(X, \cO_X((1-q)(K_X+\FRC(D))+nD)) T^n\right]^{1/q}
\]
by Serre duality. 
Then the condition $c^{1/q}(\xi\otimes 1) \ne 0$ is equivalent to say 
\[
c\cdot \bigoplus_{n\in \Z} H^0(X, \cO_X((1-q)(K_X+\FRC(D))+nD)) T^n 
\ne 0.
\] 
Since we have seen that $-K_X - \FRC(D) = - b(R) D + B$ above with 
$-b(R) D$ ample and 
$B\ge 0$, this condition is certainly satisfied and we conclude that 
$R$ is $F$-regular.  
\end{proof}

\begin{quest}\label{b(R)<0; ques} Let $R=R(X,D)$ as in \ref{3.1} and assume that $\chr(k)=0$. 
For a prime number $p>0$, let $R_p$ be a mod $p$ reduction of $R$. If $b(R)<0$ 
and $R$ has an isolated singularity,  then is $R_p$ 
$F$-pure (and $F$-regular) for sufficiently large $p$ ?
\end{quest}

\begin{rem}\label{R[t]&Segre} If $t$ is a variable over $R$, then $K_{R[t]}\cong t K_R[t]$,  
$K_{R[t]}^{-1}\cong t^{-1}\KiR[t]$ and $\Tr_{R[t]}(K_{R[t]})= \Tr_R(K_R)[t]$.  
Hence we have $b(R[t]) = b(R)- \deg(t)$   
and that if $R[t]$ is nearly Gorenstein, then $R$ is Gorenstein.
\par
Thus in Conjecture \ref{b(R)<0} and Question \ref{b(R)<0; ques} the condition \lq\lq $R$ has an 
isolated singularity" is necessary.
\par  
By the same reason, for normal graded rings $R$ and $S$ with $R_0=S_0=k$, 
$K_{R\otimes_k S}\cong K_R\otimes_k K_S$ and if $R\otimes_k S$ is nearly Gorenstein, 
then $R$ and $S$ are Gorenstein (cf. \cite{HHS}).
\end{rem}

\begin{rem} \label{Singh}
Shortly before submitting the revised version, the second-named author was informed by Anurag Singh (University of Utah) that he and Aryaman Maithani (University of Utah) had constructed normal graded rings of dimension $\ge 4$ with $b(R) < 0$ that have  isolated singularity but are not of $F$-regular type (equivalently, not log-terminal). 
We are very grateful for this information.
\end{rem}

\par
Next, we study nearly Gorenstein $2$-dimensional normal graded 
domains. 

\section{The case $\dim R =2$}
\label{dim2}

In this section, let $R=R(C,D)$ be as in \ref{3.1} and from now on, we assume 
that $k$ is {\em algebraically closed} and $\dim R=2$. \par
Then $C$ is a curve and each $V_i$ is a point, so that we can write 
\begin{equation} 
D = B - \sum_{i=1}^r (p_i/q_i) P_i,
\end{equation}
where $B$ is an integral divisor on $C$, $P_1, \ldots P_r$ are distinct points on $C$.

\par 
We define $\deg (D)$ so that $\deg(P)=1$ for every point $P\in C$, 
 since we have assumed that $k$ is algebraically closed.

\begin{noname}\label{star-resol} 
Let $\pi: X\to \Spec(R)$ be the  minimal good  resolution of $\Spec(R)$. 
Then the exceptional set $\bbE = \pi^{-1}(\m)$ is 
``star-shaped" in the sense that 
\[
\bbE = C \cup \bigcup_{i=1}^r E_{i,1}\cup E_{i,2}\cup \cdots  \cup  E_{i,s_i},
\]
with $C^2= - \deg B$, $E_{i,j}^2 = - b_{i,j}$, where each  
$E_{i,1}\cup E_{i,2}\cup \cdots  \cup E_{i,s_i}$ is a chain of $\PP^1$ 
and $b_{i,j}$ is given by the continued fraction
\[
q_i/p_i = [[b_{i,1}, b_{i,2}, \ldots , b_{i, s_i}]]. 
\] 
\end{noname}

\begin{noname}\label{pa(Z)} 
Let  $\pi: X\to \Spec(R)$ and $\bbE$ be the  
minimal good resolution and exceptional set as above. 
We call  a $\Z$-linear combination of $C$ and 
$E_{ij}$, a {\em cycle}.  
A cycle $Z$ is called {\em anti-nef} if $Z C, Z E_{i,j} \le 0$. 
Since the intersection matrix of $\bbE$ is negative definite, 
there exists  a unique minimal anti-nef cycle 
 $\Z_f  >0$. 
 Then we can show that $R$ is rational 
 (resp. {\em elliptic}) if  $p_a(\Z_f) =0$ (resp. $p_a(\Z_f)=1$), 
where $p_a(Z)$ is defined by $p_a(Z)=(Z^2+K_XZ)/2 +1$.    
\end{noname}

\par 
First we see that  the nearly Gorenstein condition  
for $R=R(C,D)$ is  very strong even if 
$R$ is a rational singularity or an elliptic singularity.
Note that if $R$ is a rational or an elliptic singularity, 
then $R$ is almost Gorenstein (cf. \cite{GTT}, \cite{OWYaG}; 
See also \S 6).
\par
We recall the characterization of nearly Gorenstein rational singularity of dimension $2$, 
using minimal resolution of $\Spec(R)$. 
This result is the main theorem of 
\cite{MOWY}.

\begin{thm}[\textrm{\cite[Theorem 4.2]{MOWY}}] \label{Zf-antinef} 
Let $(A,\m)$ be an excellent normal local ring of dimension $2$ 
and $\pi: X\to \Spec(A)$ be the minimal resolution of $\Spec(A)$. 
If $A$ is a rational singularity, then $A$ is nearly Gorenstein if and only if 
$K_X + Z_f$  is anti-nef.
\end{thm}

\begin{prop}[\textrm{cf. \cite[Theorem 4.2]{MOWY}}]
\label{R(C,D)rational}  
If $R(C,D)$ is a rational singularity, then 
$C= \PP^1$ and if $R$ is nearly Gorenstein, then 
either $r\le 2$ $($$R$ is a \lq\lq cyclic quotient singularity"$)$ 
or $\deg B = r-1$. 
\end{prop}

\begin{proof} If $\deg B \ge r$, then the fundamental cycle $\Z_f$ of the minimal resolution 
of $\Spec(R)$ is reduced and by Theorem 4.2 of \cite{MOWY}, if $R$ is nearly Gorenstein, then 
$R$ is a \lq\lq cyclic quotient singularity". 
\par 
On the other hand, if $r\ge \deg B+2$, 
then $\deg ([D]) = \deg B - r\le -2$. 
Then we have $[H^1_{\m}(R)]_{1}\ne 0$. 
Hence we have $a(R) > 0$ and $R$ is not a rational singularity. 
\end{proof}

\par 
The condition $\deg B = r-1$ is very far from $R$ to be nearly Gorenstein. 
We will compute the nearly Gorenstein property of $R(C,D)$ of a special type.

\begin{ex} For an integer $r\ge 3$, let $R=R(C,D)$ with 
\[D = (r-1)Q - \dfrac{p}{q} (P_1+P_2+\ldots + P_r),\]
where $C=\PP^1, Q, P_1, P_2, \ldots,  P_r$ are points on $\PP^1$ and we assume 
$P_1, P_2, \ldots,   P_r$ are different points and 
$\deg D = (r-1)- r p/q >0$. Then we have the following results:
\begin{enumerate}
\item $\dim_k [K_R]_1 = r-2 >0$.
\item $R$ is a rational singularity if and only if $p/q \le 1/2$. 
\item If $R$ is a rational singularity then $R$ is nearly Gorenstein
 if and only if $p=1$ and $R$ is Gorenstein; 
that is, $p=1, q=2$ and $r=3$. 
\item  $R$ is an elliptic singularity if and only if 
$1/2 < p/q \le 2/3$. In this case, 
we have $m(R)=3$ and $\dim_k R_3 = r-2$.  
\item If $R$ is an elliptic singularity and if $r\ge 4$, 
then we have $a(R)=2$ and $p_g(R)=1$. If $R$ is an elliptic singularity, 
$r=3$ and $(1+2s)/(2+3s) < p/q \le (3+2s)/(5+3s)$ for 
$s\in \Z_{\ge 0}$, then $\deg[nD]=-2$ for 
$n= 2, 5, \cdots , 2+3s$ and $\deg[nD]\ge -1$ for other $n\ge 0$. 
Then we have $a(R)= 2+3s$ and $p_g(R)= s+1$.  
\item  If $q$ is odd and $p=(q+1)/2$ then $R$ is Gorenstein with $a(R)=2$.  
Conversely, if $R$ is Gorenstein with $a(R)=2$, then we have $2p=q+1$. 
\item If $r=3$, then for $s=0,1,2,\ldots$,  $R$ is Gorenstein with 
$a(R)= 2+3s$, if and only if $(2+3s)p=(1+2s)q+1$ holds. 
\item  If $p/q = 5/8$ and $r\ge 4$, then $R$ is elliptic, nearly Gorenstein  
and not Gorenstein.  
\end{enumerate} 
\end{ex} 

\begin{proof} 
(1) It is easy to see that $[K_C + \FRC(D) + D] \sim (r-3)Q$. Since 
$r-3\ge 0$, we deduce that $[K_R]_1\ne 0$. 
\par \vspace{1mm}
(2) We have always $\deg [D] = -1$.  We have $\deg [2D] = r-2>0$ if $p/q \le 1/2$ 
and $-2$ if $p/q > 1/2$. Hence $[H^2_{\m}(R)]_2 \ne 0$ and $a(R)>0$ if $p/q > 1/2$, which 
implies that $R$ is {\em not} a rational singularity. Conversely, if $p/q \le 1/2$, then it is 
easy to see that $\deg [nD] \ge 0$ for every $n\ge 2$ and then we have $a(R)<0$ and $R$ is   
a rational singularity. 
\par \vspace{1mm}
(3) We assume that $p/q \le 1/2$. Then $a(R)=-1$ and $m(R)=2$ by (1), (2). 
Hence if $R$ is nearly Gorenstein, then we must have $b(R)=1$ by Proposition \ref{a(R)b(R)}.
We compute  $\deg ([-K_C -\FRC(D) + D ]) = 1$ if $p=1$ and $1 - r <0$ if $p > 1$ since 
$-(q-1)/q - p/q] = -2$ if $p>1$. Namely, $[\KiR]_1\ne 0$ if and only if $p=1$. 
Thus $R$ is not nearly Gorenstein if $p>1$. 
\par 
If $p=1$, then the resolution of $\Spec(R)$ as in 
\ref{star-resol} is $\bbE = C \cup \bigcup_{i=1}^r E_i$ with $C^2 = -r+1, E_i^2=-q$ 
for every $i$. Then it is easy to see that $Z_f = 2C + \sum_{i=1}^r E_i$ and 
$Z_f + K_X$ is anti-nef. Hence by Theorem \ref{Zf-antinef}, 
$R$ is nearly Gorenstein if $p=1$. 

\par \vspace{1mm}
(4) If $p/q > 1/2$, we have seen that $\deg ([2D]) =-2$ and similarly, if $p/q > 2/3$, 
then $\deg([3D]) = -3$ and $R$ is not an elliptic singularity by Proposition \ref{Tomari}. 
On the other hand, if $p/q \le 2/3$, then we have $\deg([3D]) = r-3\ge 0$ and we have 
$\dim R_3 = r-2 > 0$ and  $m(R) = 3$. Then $R$ is an elliptic singularity by Proposition 
\ref{Tomari}. 
\par \vspace{1mm}
(5) If  $1/2< p/q \le 2/3$, then we have  $\deg([4D]) = 4(r-1) - 3r = r-4$. 
Also, $\deg([5D]) = 5(r-1)-3r = 2r -5 >0$ if $p/q \le 3/5$ and 
$\deg([5D]) = 5(r-1)-4r = r-5$ if $p/q > 3/5$. Thus if $r \ge 4$, then 
$\deg ([nD])\ge -1$ for all $n\ge 3$ and thus $a(R)=2$. 
\par 
If $r=3$ and if $(1+2s)/(2+3s) < p/q \le (3+2s)/(5+3s)$, then we can see that 
$\deg[ (3s+2) D] = 2 (3s+2) - 3(2+2s) = -2$ and  $\deg[ (3s+5) D] >0$ 
and we have $a(R) = 3s+2$.  In this case, we have 
$\dim_k [H^2_{\m}(R)]_n = 1$ for 
$n = 2,5, \ldots , 3s+2$ and we have $p_g(R)=s+1$.  
\par \vspace{1mm}
(6) We compute 
\begin{eqnarray*}
K_C + \FRC(D) - 2D 
& \sim & -2Q + \left( \dfrac{q-1}{q}+ \dfrac{2p}{q}\right) (P_1+\ldots + P_r) - 2(r-1)Q  \\[1mm]
&\sim & \dfrac{2p-q-1}{q}(P_1+\ldots + P_r). 
\end{eqnarray*}
From this we see that $R$ is Gorenstein with $a(R)=2$ if and only if $2p=q+1$. 
\par \vspace{1mm}
(7) If $(2+3s)p=(1+2s)q+1$ holds, then we can check that $(2+3s)D \sim K_C + \FRC(D)$. 
Conversely, if $(2+3s)D \sim K_C + \FRC(D)$, we have $(2+3s)p=(1+2s)q+1$. 
\par \vspace{1mm}
(8) If $p/q = 5/8$ and $r\ge 4$, then we see that 
$5D \sim (r-5)Q +\FRC(D) \sim (K_C +\FRC(D)) + (r-3) Q$. 
We can see that $R$ is generated by $R_3, R_4, R_8$
 (and $R_5$ if $r\ge 5$) 
 and we can show that for these $R_n$,  
 $R_n = [K_R]_2\cdot [\KiR]_{n-2}$ for $n\ne 8$ and $R_8 =  [K_R]_3\cdot [\KiR]_5$. 
\end{proof}

\begin{rem} 
Among the fractions which satisfy $1/2 < p/q \le 2/3$ and does not satisfy $2p=q+1, 5p=3q+1$ the one with the smallest $q$  is $10/17$. 
If we put $D= (r-1)Q -10/17 (P_1+ P_1+\ldots + P_r)$, 
then we have $[5D] \sim (2r-5)Q$ and 
$[-K_C - \FRC(D) + 7D ] \sim (r-5)Q$.  This means that $[\KiR]_7\cdot [K_R]_{-2}
\subsetneq R_5$ and since $[\KiR]_n =0$ for $n\le 4$, this shows that 
$R$ is not nearly Gorenstein. \par
This suggests that $p/q$ which gives nearly Gorenstein elliptic singularities in our setting is rather restricted.    
\end{rem}

\par 
We note here the criterion for $R(C,D)$ to be 
 an elliptic singularity given by M. Tomari. 

\begin{prop}[\textrm{Tomari \cite[Corollary 3.9]{tomari.max}}]\label{Tomari}
Let $R=R(C,D)$  be as above. 
For every $n\in \Z_{\ge 0}$, we denote by $[nD]$ the  integral part of $nD$. 
Let $A$ be the localization of $R(C,D)$ with respect to $\m$. 
Then $A$ is an elliptic singularity if and only if  one of the following conditions holds:
\begin{enumerate}
\item $C$ is an elliptic curve and $\deg [D] \ge 0$.
\item $C$ is a rational curve and there exist integers 
$1\le m_1 < m$ such that 
$m$ is the minimum of positive integers with $\deg [mD] \ge 0$, 
$\deg [m_1D] = -2$ and $\deg [iD] = -1$  for $1 \le i <m$ with $i\ne m_1$.
\end{enumerate}
Note that the conditions above depend only on the genus $g=g(C)$ and the degrees of $[nD]$ 
$($$n\ge 1$$)$.
\end{prop}

\par 
We study normal graded  rings $R(C,D)$ which are 2-dimensional elliptic singularities. 

\begin{prop}\label{ell1}
Let $R=R(C,D)$, where $C$ is an elliptic curve, 
$D = B - \sum_{i=1}^r (p_i/q_i)P_i$ as 
in equation $(4.1)$ and we assume that $r>0$ and $\deg B > r$. 
If $R$ is nearly Gorenstein, then $p_i=1$ for every $1\le i \le r$. 
\end{prop}

\begin{proof}  
Then $\dim R_1 = \deg B - r >0, [K_R]_0=k$ and 
it is easy to see that $[K_R]_n = 0$ for $n< 0$. 
Hence we have $a(R)=0$ and $m(R)=1$. 
Hence if $R$ is nearly Gorenstein,  by Proposition \ref{a(R)b(R)}, 
we must have 
\[
b(R)=1 \quad {\rm and} \quad \dim_k R_1 = \dim_k [\KiR]_1= 
\dim_k H^0(\cO_C([-\FRC(D) +D])).
\] 
Now, $[-(q_i -1)/q_i - p_i/q_i]= -1$ or $-2$ and $-1$ is attained only if $p_i=1$. 
Thus the condition 
$\dim_k R_1 = \dim_k [\KiR]_1 = \dim_k H^0(\cO_C([-\FRC(D) +D]))$ 
implies $p_1=1$ for every $i$. 
\end{proof}

\par  
We think that there are very few normal graded rings which are 
non-Gorenstein nearly Gorenstein elliptic singularities. 
We give one such example.

\begin{ex}  
Let $C$ be an elliptic curve and 
let $P_1,\ldots, P_r$ ($r\ge 3$) be different points of $C$. 
Put $B=P_1+\cdots + P_r$ and let 
\[ 
D = 2B -\dfrac{1}{2} B 
\]
Then $R=R(C,D)$ is nearly Gorenstein but not Gorenstein.
\par
Actually, if we put $\cA = R(C,B)$, then $\cA$ is a Gorenstein graded ring generated by $\cA_1$.  
Moreover,  $R = k[ \cA_1T, \cA_3T^2]$ and we can see $R_1= [K_R]_0\cdot [\KiR]_1$ and 
$R_2 = [K_R]_1\cdot [\KiR]_1$. 
\end{ex}

\medskip
\section{2-dimensional \lq\lq cone singularities"}
\label{s:cone}

In this section, let $R=R(C,D)$, where $C$ is a smooth curve of genus $g$ and $D$ is an {\bf integral} ample divisor on $C$.  
We call such $R(C,D)$ \lq\lq cone singularity". 
\par 
If $g=0$ and $\deg(D)=d$, then $R$ is the $d$-th Veronese 
 subring of a polynomial ring $k[X,Y]$ and is nearly Gorenstein (\cite{HHS}), and if $g=1$, 
 then $R(C,D)$ is Gorenstein since it is a \lq\lq simple elliptic singularity". 
 So we assume $g\ge 2$ in the following.  
Also we assume that {\bf $k$ is algebraically closed}. 
\par 
Moreover, we write $D\sim D'$ for $D, D'\in \Div(C)$ if $D - D' = \di_C(f)$ for some $f\in k(C)$,
or, equivalently, $\cO_C(D)\cong \cO_C(D')$. 
For $D \in \Div(C)$, we write 
\[
h^i(D)= h^i(C,\cO_C(D))=\dim_k H^i(C,\cO_C(D)) \;\;\;\text{for $i=0,1$}. 
\]

\par 
We recall \lq\lq Riemann-Roch Theorem", 
which plays very important role in this 
section. 

\begin{thm}\label{RRTh} Let $D\in \Div(C)$. Then we have  
\begin{enumerate}
\item $($\textbf{Riemann-Roch Theorem}$)$
\[
h^0(D)-h^1(D)=\deg(D) +1-g.
\]
\item $($\textbf{Serre Duality}$)$ $h^1(D) = h^0(K_C-D)$, where $K_C$ is the canonical bundle of $C$. 
Note that $\deg(K_C)= 2g-2$, $h^0(K_C)=g$ and $h^1(K_C)=1$. 
\item If $\deg D = 2g-2$, then $h^0(D)\ge g-1$, 
and $h^0(D)=g$ if and only if $D\sim K_C$.
If $\deg D \ge 2g-1$, then $h^0(D)= \deg D - g+1$.  
\item When $g \ge 1$, $h^0(D) \le \deg(D)$. 
\end{enumerate}
\end{thm}

\par 
First, we show that if $\deg(D)$ is very big relative to $g$, then $R(C,D)$ is nearly 
Gorenstein. Recall that we showed in 
Proposition \ref{a(R)b(R)} (2), (3) if $R=k[R_1]$,  
then $R$ is nearly Gorenstein but not Gorenstein if and only if 
$-a(R)+b(R)=1$ and 
 $R_1= [K_R]_{-a(R)}\cdot [\KiR]_{b(R)}$. 
Since $R(C,D)$ is generated by $R_1$ if $\deg D\ge 2g+1$ by \cite{Mum}, 
we have the following result by Proposition \ref{a(R)b(R)} (2), (3).

\begin{cor}\label{ge 2g+1} If $\deg(D)\ge 2g+1$, then $R(C,D)$ is nearly Gorenstein 
if and only if the product mapping 
\[H^0(C, \cO_C(K_C))\otimes_k H^0(C, \cO_C(D - K_C)) \to H^0(C, \cO_C(D))\]
is surjective. 
\end{cor} 
  
\begin{thm}\label{high-deg} 
Let $C$ be a smooth curve of genus $g\ge 2$ and 
$D$ be an integral ample divisor on $C$. 
We put $R = R(C,D)$ and assume 
that $\deg D > 2g -2 = \deg K_C$.
\par
\begin{enumerate}
\item If $H^1(C, \cO_C(D - 2 K_C))=0$, then 
we have $H^0(C, \cO_C(D))= H^0(C, \cO_C(D-K_C))\cdot H^0(C, \cO_C(K_C))$.
In particular, if $H^1(C, \cO_C(D - 2 K_C))=0$, then $R=R(C,D)$ is nearly Gorenstein.  
\item If $\deg D > 6g -6$, then $R$ is nearly Gorenstein. 
\item If $\deg D = 6g-6$, then $R$ is nearly Gorenstein except for the 
case $g=2$ and $D \sim 3K_C$. 
\item If $\deg D > 3g -3$, then $R_n = [\KiR]_n [K_R]_0$ for every $n\ge 2$. 
Namely, $\Tr_R(K_R) \supset \m_R^2$ in this case.  
\end{enumerate}
\end{thm} 

\begin{proof} 
Let $w_1, w_2$ be general elements of $H^0( C, \cO_C(K_C))$.
Then since $w_1, w_2$ generate $\cO_C(K_C)$, the mapping 
 $\phi : \cO_C^{\oplus 2} \to \cO_C(K_C) $ defined by $(w_1,w_2)$ is surjective and 
 we have an exact sequence of $\cO_C$-modules
\[ 0 \to \cO_C( nD - 2K_C)) \to \cO_C(nD - K_C)^{\oplus 2} \overset{\phi_n}\to \cO_C( nD) \to 0\]  
for every $n \in \Z$. 
In particular, if $H^1(C, \cO_C(nD-2K_C))=0$, then we have 
$H^0(C, \cO_C(nD)) = (w_1, w_2)H^0(C, \cO_C(nD - K_C))$, that is, we have $R_n = [\KiR]_n [K_R]_0$.  
\par \vspace{2mm}
(1) Suppose $H^1(C, \cO_C(D-2K_C))=0$. 
Then $h^0(3K_C-D)=0$ by Serre duality. 
Theorem \ref{RRTh}(3) implies $\deg(3K_C-D) \le 2g-2$, that is, 
$\deg(D) \ge 4g-4$. 
Then $H^1(C,\cO_C(nD-2K_C))=0$ for all $n \ge 1$.  
By the argument as above, $R_n=[K_R]_0[\KiR]_n \subset \Tr_R(K_R)$ for all $n \ge 1$. 
Hence $R$ is nearly Gorenstein. 
\par  \vspace{1mm}
(2) If $\deg D > 6g -6$, then since $\deg(D - 2K_C) > 2g-2$, 
one can easily see that $H^1(C,\cO_{C}(D-2K_C))=0$. 
Hence $R$ is nearly Gorenstein. 
\par  \vspace{1mm}
(3) Assume that $\deg D=6g-6$. Then, since $\deg (D- 2K_C) = 2g-2$, \par\noindent
$H^1(C, \cO_C(nD - 2K_C))\ne 0$ if and only if $n=1$ and $D \sim 3 K_C$ and otherwise,  
$R_n= [K_R]_1\cdot [\KiR]_{n-1}$.   If $D \sim 3 K_C$ and $C$ is not hyperelliptic, 
then since $R(C, K_C)$ is a standard graded Gorenstein ring, 
$R=R(C, K_C)^{(3)}$ is nearly Gorenstein by  Proposition \ref{R(d)stand}. 
Then assume $D \sim 3 K_C$ and $C$ is hyperelliptic. If $g\ge 3$, then $6g-6 \ge 4g-1$, 
$R$ is nearly Gorenstein by   Proposition \ref{hypell} (3) and if $g=2$, 
$R$ is not nearly Gorenstein by  Proposition \ref{hypell} (2). 
\par  \vspace{1mm}
(4) If   $\deg D > 3g -3$, since $\deg nD > 6g-6$ for $n\ge 2$, we have 
$R_n = [\KiR]_n [K_R]_0$ for every $n\ge 2$. 
\end{proof}

\par 
Assume that $R=R(C,D)$ is Gorenstein. 
We ask if the Veronese subring 
$R^{(d)}= \bigoplus_{n\ge 0} R_{dn}$ is Gorenstein.  
If $R=k[R_1]$, the answer is given in \cite{HHS}.   

\begin{prop}\label{R(d)stand} If $R= \bigoplus_{n\ge 0} R_{n}$
 is a Gorenstein graded ring and if 
$R=k[R_1]$, then for every $d\ge 2$, $R^{(d)}$ is nearly Gorenstein. 
\end{prop}

But if $R$ is not a standard graded ring, the result is different. 

\begin{ex}\label{235+237}[\textrm{cf. \cite{MOWY}}] Let $R= k[X,Y,Z]/(X^2+Y^3+Z^5)$ with $\deg(X,Y,Z) = (15,10,6)$. 
\begin{enumerate}
\item[\rm{(a)}] If $d$ is divisible by either $2$, $3$, $5$, then $R^{(d)}$
is a cyclic quotient singularity and hence is nearly Gorenstein. 
\item[\rm{(b)}] If $d$ is relatively prime with $2$, $3$ and $5$, then $R^{(d)}$ is nearly 
Gorenstein if and only if $d = 1, 7, 11, 17, 19, 29$.  
\end{enumerate}

\par 
The proof is reduced to compute the fundamental cycle of the minimal resolution of $R^{(d)}$. 
Note that if $d>60$ and if $60s < d \le 60(s+1)$, 
then the central curve of 
the minimal good resolution of $\Spec(R^{(d)})$ is a $-(2+s)$ curve.
\end{ex}

\par 
We will give a classification of $D$ such that $R(C,D)$ is nearly 
Gorenstein in the case 
$g = 2,3$.  
For that purpose, the following $2$ results are very important. 
\par
The question \lq\lq when is $R^{(d)}$ nearly Gorenstein" gives very 
complicated and interesting answer.  

\begin{prop}\label{hypell} 
Let $g\ge 2$ and $R= k[X,Y,Z]/(Z^2 + f_{2g+ 1}(X^2,Y))$ with degree 
$\deg (X,Y,Z) = (1,2, 2g+1)$, where $f_{2g+1 }(X^2,Y)$ is a 
homogeneous polynomial of 
degree $2g+1$ of $(X^2, Y)$ with no multiple roots.
Note that $a(R)= 2g-2$ in this case. 
This is a cone singularity $R(C,P)$, where 
$C$ is a hyperelliptic curve of genus $g$ and $P\in C$ is a hyperelliptic point 
$($$\dim_k( H^0(C, \cO_C(2P))=2)$. 
Note that $R^{(d)}$ is Gorenstein if and only if $d \,|\, (2g-2)$.   
\par
In the following, we always assume that $d$ does not divide $2g-2$. 
\begin{enumerate}
\item[(1)] If $d=3$ or $4$, then $R^{(3)}, R^{(4)}$ is nearly Gorenstein for every $g$.
\item[(2)] If $d=6$, then $R^{(6)}$ is nearly Gorenstein 
if and only if $g \not\equiv 2 \pmod{3}$. 
\item[(3)] If $d\ge 4g-1 = \deg K_C + \deg Z$, then $R^{(d)}$ is nearly Gorenstein. 
\item[(4)] In the following, we assume that $d\ge 5$. 
\begin{enumerate}
\item[{\rm (a)}] If $g=2$, then $R^{(d)}$ is nearly Gorenstein if and only if $d\ge 7$.
\item[{\rm (b)}]  If $g=3$, then $R^{(d)}$ is nearly Gorenstein if and only if $d=5, 6$ or 
$d\ge 11$.
\item[{\rm (c)}] 
 If $g=4$, then $R^{(d)}$ is nearly Gorenstein if and only if 
$d =5,7, 8$ or $d\ge 15$.
\item[{\rm (d)}]  If $g=5$, then $R^{(d)}$ is nearly Gorenstein if and only if 
$d \le 10, d\ne 6$ or $d\ge 19$.
\end{enumerate}
\end{enumerate}
\end{prop}

\begin{proof} 
Fix $d$ and put $L=K_R^{-1}$. 
\par
Let $K^{(d)}_n$ be the degree $n$ part of 
$K_{R^{(d)}}= [K_R]^{(d)}$, that is, 
\[
K^{(d)}_n=H^0(C, \cO_C(K_C + (nd)P)) = 
H^0(C, \cO_C((2g-2+ nd)P)). 
\]
Let  $L^{(d)}_n$  be the degree $n$ part of 
$K^{-1}_{R^{(d)}}$, that is, 
\[
L^{(d)}_n=H^0(C, \cO_C(-K_C + (nd)P)) = 
H^0(C, \cO_C((-2g+2+ nd)P)). 
\]

\par
(1) First suppose $d=3$. 
If $3 \not | \; 2g-2$, we can find $n$ such that $(2g-2) - 3n = 1$ or $2$, so that
$ K^{(3)}_{-n} \cdot L^{(3)}_{n+1} = R_3$ and similarly for $R_6$. 
\par
Note that if  $2g+1 = \deg Z$ is divisible by $3$, then so is $2g-2$ and hence 
$R^{(3)}$ is Gorenstein. Then it is easy to see that for every $m$, 
we can find $n$ so that $R_{3m} = K^{(3)}_n\cdot  L^{(3)}_{m-n}$. 
\par \vspace{1mm}
Next suppose $d=4$. If $g$ is odd, since $2g-2$ can be 
divided by $4$, then $R^{(d)}$ is Gorenstein. 
If $g=2m$ is even, then we can see $K^{(4)}_{-m+1}\cdot L^{(4)}_{m} = (R_2)^2 =R_4$.
Next, $m+1$ is the smallest degree $n$ such that $Z$ appears in 
$ R^{(4)}_{n} = R_{4n}$ and 
we can see that $K^{(4)}_{1}\cdot L^{(4)}_{m}= R^{(4)}_{m+1}$.  

\par \vspace{1mm}
(2) If $g \equiv 2 \pmod{3}$, then $2g + 1 = 6m -1$ for some positive integer $m$ 
and $XZ \in R^{(6)}_m = R_{6m}$.  But if $s= \pm (2g-2) + 6n < 6m$, then $R_s$ does not 
contain $Z$ and hence $XZ$ does not appear in $K^{(6)} \cdot  L^{(6)}$ and so $R^{(6)}$ is not nearly Gorenstein.  

\par \vspace{1mm}
(3) If $d \ge 4g-1 = \deg Z + \deg K_C$, then it is easy to see that 
$R_d = K^{(d)}_{0}\cdot L^{(d)}_{1}$. Since   $R^{(d)}$ is generated by $R_d$, 
$R^{(d)}$ is nearly Gorenstein.  
\par 
Note that if $g \equiv 1 \pmod{3}$ then $6 \,|\, 2g-2$.  
\par \vspace{1mm}
(4) (a) If $g =2$ and $d=5$, then $Z \in R_5$ is not contained in 
$K^{(5)} \cdot  L^{(5)}$
and $R^{(5)}$ is not nearly Gorenstein. 
Similarly, $R^{(6)}$ is not nearly Gorenstein because 
$ZX \in R_6 \setminus  K^{(6)} \cdot  L^{(6)}$. 
If $d\ge 7$,  then the assertion follows from (3). 

\par 
(b) If $g=3$, then $\deg K_C =4$ and $\deg Z =7$. If $d=5$, then 
$R_5 = K^{(5)}_0 \cdot  L^{(5)}_1$ and 
$R_{10} = K^{(5)}_1 \cdot  L^{(5)}_1 + 
K^{(5)}_0 \cdot  L^{(5)}_2$ imply that $R^{(5)}$ is nearly Gorenstein. 
We can show that $R^{(7)}$ (resp. $R^{(8)}$) is not nearly Gorenstein, 
since $Z \in R_7$ (resp. $XZ \in R_8$) is not contained in 
$K^{(8)} \cdot  L^{(8)}$. 
(resp. $K^{(8)} \cdot  L^{(8)}$). $R^{(9)}, R^{(10)}$ is not nearly Gorenstein in the same manner. 
\par  
(c) If $g=4$, then  $R$ is generated by $X,Y,Z$ with degrees $1,2, 9$ and $\deg (K_C) = 6$.  
We consider the cases $d = 5,7,8, \ldots$.
\par
If $d = 4,5,7,8$, we can show that $R^{(d)}$ is nearly Gorenstein. 
If $9\le d \le 14$, we put  
$d = 9+s$. 
Then we can easily see that $X^s Z\in R_d$ is not in  $K^{(d)} \cdot  L^{(d)}$.
Hence  $R^{(d)}$ is not nearly Gorenstein. If $d \ge 15$, then $R^{(d)}$ is nearly Gorenstein
by (3). 
\par 
(d) If $g=5$, then  $R$ is generated by $X,Y,Z$ with degrees $1,2, 11$ and $\deg (K_C) = 8$.
We can have the result by the same argument as in (a), (b), (c). 
\end{proof}

\begin{prop}\label{degD|2g-2} 
Assume that $\deg (D)$ divides $2g-2 = \deg (K_C)$ and that\par
 $H^0(C, \cO_C(D)) \ne 0$. 
Then $R(C,D)$ is either Gorenstein or not nearly Gorenstein.  
\end{prop}
\begin{proof} 
This is a special case of Theorem \ref{degD-div}. 
\end{proof}

We use this Proposition for $R(C,D)$, for some special values of $\deg(D)$.

\begin{prop}\label{d=g-1} We assume $g\ge 2$, $R=R(C,D)$ and assume that 
$R$ is not Gorenstein. Then
\begin{enumerate}
\item If $\deg(D)= g-1$, then $R$ is not nearly Gorenstein. 
\item If $g\ge 3$, $g\le \deg(D) < 2g-2$ and if $R$ is nearly Gorenstein, then we have 
$a(R)=1, b(R)=2$ and $R_1 = [K_R]_{-1}\cdot [\KiR]_2$.   
\item If $\deg(D) = 2g-3 = \deg (K_C) -1$ and if $R$ is nearly Gorenstein, then 
$C$ is a hyperelliptic curve and $D \sim K_C - P \sim (2g-3)P$, 
where $P$ is a hyperelliptic point of $C$. 
If this is the case, then $R$ is nearly Gorenstein. 
\item If $\deg D = 2g-1$ and if $R(C,D)$ is nearly Gorenstein, then 
$D\sim K_C +P$ for some point $P\in C$. 
Conversely, if $D = K_C + P$ for some point $P\in C$, 
then $R(C,D)$ is nearly Gorenstein. 
\end{enumerate}
\end{prop}

\begin{proof} (1) Assume on the contrary, $R$ is nearly Gorenstein. 
Then by Proposition \ref{degD|2g-2}, we must have $h^0(D) =0$ and $m(R)=2$.
But then by Riemann-Roch Theorem since $h^0(D) - h^1(D)=0$ and 
$h^1(D)=h^0(K_C-D)$, we have $h^0(K_C -D)=0$, which implies $a(R)=0$. 
Since $h^0(-K_C+2D) =0$, we have $b(R)=3$ and $b(R) -a(R)=3 > m(R)$. 
Hence $R$ is not nearly Gorenstein.\par  
(2) This follows from Proposition \ref{a(R)b(R)} since $h^0(D)>0$.  
\par
(3) We assume $g\ge 3$ and  $2g-3 > g-1$, since if $g=2$ we can reduce to the case (1). 
Since $\deg(K_C-D)=1$ and $a(R)=1$, we must have $h^0(K_C-D))=\dim_k ([K_R]_{-1})=1$ and  
by Riemann-Roch formula, $h^0(D) - h^0(K_C -D) = \deg(D) - (g-1) = g-2$, we have   
$\dim_k R_1 = g-1$. Also, we must have $D=K_C -P$ for some point 
$P\in C$. 
By (2), we must have $\dim_k([\KiR]_2) = g-1$. \par
If $C$ is hyperelliptic, $D\sim K_C-P$ for a hyperelliptic point $P$, 
then since $K_C \sim (2g-2)P$, we have $D\sim (2g-3)P$. 
 Hence the nearly Gorenstein property of $R(C,D)$ is reduced to Proposition  
\ref{hypell}, $d=2g-3$. The case $g\le 5$, $R$ is shown to be nearly Gorenstein there.
If $g\ge 6$, putting $S = k[X,Y,Z]/(Z^2 - f_{2g+1}(X^2,Y))$, we can check 
 $Z S_{2g-7}\subset R_2$ is included in $[K_R]_{-1}\dot [\KiR]_3$ and $R$ is nearly Gorenstein.  
 \par 
(4) We put  $L= \KiR$. If $\deg D = 2g-1$ and if $R(C,D)$ is nearly Gorenstein, then we must have 
$R_1= K_0\cdot L_1$ and $L_1 = H^0(D-K_C)\ne 0$. Since $\deg (D-K_C)=1$, we must have 
$D-K_C \sim P$ for some point $P\in C$. 
\par
Conversely, assume that $D=K_C +P$. Then note that by Riemann-Roch Theorem, 
 $h^0(D) - h^1(D) = g$ and $h^1(D)=0$, we have $\dim R_1= h^0(D)=g$. Since $\dim L_1=1$, 
 we have $R_1= K_0\cdot L_1$.  
 Also, for $n\ge 3$, we have $R_n = K_0\cdot L_n$ by 
 Theorem \ref{high-deg} (1). Hence  to show that $R$ is nearly 
 Gorenstein, it suffices to show  $R_2 = K_1\cdot L_1 + K_0\cdot L_2$. 
Now, since $L_1 = H^0(C, \cO_C(P)) = k$, 
$L_1 \cdot K_1=H^0(2D-P)$ is a subspace of $R_2$ of codimension $1$. 
Hence we can assert $L_1\cdot K_1 + L_2\cdot K_0 =R_2$ 
if $L_2\cdot K_0$ is not included in 
$H^0(2D -P)$.  Now take $\phi \in H^0(C, \cO_C(D+P))$ so that 
$H^0(C, \cO_C(D+P))$ 
is generated by $H^0(C, \cO_C(D))$ and $\phi$. Then we can assert that 
$\phi\cdot H^0(C, \cO_C(K_C))$ is not included in $H^0(2D-P)$ and then we can conclude 
that $L_1\cdot K_1=H^0(2D-P)$. 
\end{proof}

\par
We classify the nearly Gorenstein cone singularities $R(C,D)$ when genus $g(C)$ of $C$ is $2$ and $3$. 

\begin{thm}\label{g=2} 
Let $C$ be a smooth curve of genus $g=2$ and $D$ be an integral  ample divisor on $C$.
We always assume that $R=R(C,D)$ is not Gorenstein. 
\begin{enumerate}
\item If $\deg D=1$ or $2$, then $R$ is not nearly Gorenstein. 
\item If $\deg D =3$, $R$ is  nearly Gorenstein
 if and only if $D \sim K_C + P$ for some point $P\in C$.
\item If $\deg D=4$, $R$ is nearly Gorenstein if and only 
if $D \sim 2K_C$. 
\item If $\deg D = 5$, then $R$ is  nearly Gorenstein if and only if 
\par\noindent
$H^0(C, \cO_C(3K_C -D))=0$. 
\item If $\deg D = 6$, then $R$ is nearly Gorenstein 
if and only if $D \not \sim 3K_C$. 
\item If $\deg D \ge 7$, then $R$ is nearly Gorenstein. 
\end{enumerate}
\end{thm}

\begin{proof} (1)
If $\deg D = 1 =g-1$, $R$ is
  not nearly Gorenstein by Proposition \ref{d=g-1}. If $\deg D =2$, since $h^0(D)\ge 1$, 
 if $R$ is nearly Gorenstein, then $R$ is Gorenstein by Proposition \ref{degD|2g-2}. 
\par 
Now suppose that $\deg D \ge 3$. 
Then $R$ is not Gorenstein since $\deg D > \deg K_C$, 
$a(R)=0$ and $\dim_k R_1 = \deg D -1$. Hence if
$R$ is nearly Gorenstein, then we must have $m(R)=1$, $a(R)=0$ and $b(R)=1$. 
Then we must have $[\KiR]_1= H^0(C, \cO_C(D-K_C))\ne 0$. 
\par
(2) This follows from Proposition \ref{d=g-1} (4). 
\par
(3) If $\deg D =4$, then since $\dim_k R_1=3$ and 
$\dim_k K_0 =2$, we must have 
$\dim_k [\KiR]_1 \ge 2$, which implies that $D -K_C \sim K_C$. Conversely, if $D\sim 2K_C$, then 
$R$ is nearly Gorenstein by Proposition \ref{hypell} (1).

\par 
\vspace{1mm}
(4)  Let $\deg D =5$. Then we have $\dim R_1 =4$. 
 Note that $H^1(C,\cO_C(D-K_C))=0$.
 Again, we consider the exact sequence as in the proof of Theorem 
\ref{high-deg};  
\[ 0 \to \cO_C(D-2K_C)) \to \cO_C( D- K_C)^{\oplus 2} \to \cO_C(D ) \to 0.\]  
This implies that $[K_R]_0\cdot [\KiR]_1 =R_1$ holds if and only if $H^1(C, \cO_C(D - 2K_C))=0$, 
which is equivalent to say that $H^0( \cO_C(3K_C -D))= 0$.

\par \vspace{1mm}
(5) and (6) are proved in Theorem \ref{high-deg}.
\end{proof}
Next we treat the case $g=3$.

\begin{thm}\label{g=3} 
Let $C$ be a smooth curve of genus $g=3$ and 
$D$ be an  integral ample divisor on $C$. 
We always assume that $R=R(C,D)$ is not Gorenstein.
\begin{enumerate}
\item  If $\deg D=1$ or $2$, then $R$ is not nearly Gorenstein.
\item If $\deg D =3$, $R$ is nearly Gorenstein if and only if $C$ is hyperelliptic and 
$D\sim 3P$ for some hyperelliptic point $P\in C$. 
\item If $\deg D = 4=\deg K_C$,  
then $R$ is not nearly  Gorenstein.  
\item If $\deg D =5$, then $R$ is nearly Gorenstein if and only if 
$D \sim K_C + P$ for some point $P\in C$. 
\item If $\deg D =6$, then $R$ is nearly  Gorenstein if and only if 
$C$ is hyperelliptic and $D\sim 6P$ for a hyperelliptic point $P\in C$
 (the case of $g=3$ of Proposition \ref{hypell} (2)). 
\item If $\deg D = 7$, $R$ is nearly Gorenstein if and only if 
$D = K_C + B$, 
where $B$ is a divisor on $C$ with $\deg B = 3$,  
$h^0(B)=2$ and $\cO_C(B)$ has no fixed point.  
\item If $\deg D = 8$, $R$ is nearly Gorenstein if and only if 
$D \sim  2 K_C$ and $C$ is not hyperelliptic. 
\item If $9 \le \deg D \le 11$, then $R$ is nearly Gorenstein if and only if 
\[
R_1 = H^0(C,\cO_C(D))= H^0(C, \cO_C(K_C))\cdot H^0(C, \cO_C(-K_C + D)).
\] 
The explicit condition for this equality to hold is not known to us yet.  
\item If $\deg D \ge 12$, then $R$ is nearly Gorenstein. 
\end{enumerate}
\end{thm}
\begin{proof} 
During the proof of this theorem, we put 
\[K_n = [K_R]_n = H^0(C, \cO_C(K_C+nD))\quad {\rm and} \quad  L_n = [\KiR]_n = 
H^0(C, \cO_C(- K_C+nD)).
\] 
By Riemann-Roch Theorem and Serre Duality Theorem, we always have 
\[
(*) \quad h^0(D)-h^0(K_C-D)= \deg D - 2.
\]
Hence we always have 
\[
(**) \quad \dim_k R_n = \dim_k K_{-n} + n \deg D -2.
\] 

\par \vspace{1mm}
(1) Suppose that $R$ is nearly Gorenstein, and let $\deg D =1$.  By Proposition \ref{degD|2g-2}, we have $R_1 =0$. This implies that 
$\dim K_{-1}=1$ by $(**)$ above. 
Also, we have $\dim_k R_2 = \dim_k K_{-2}$ and $\dim_k R_3 
= \dim_k K_{-3} +1$, 
$K_n = 0$ for $n\le -4$ and $L_n =0$ for $n\le 4$.  

\par 
(a) First we assume that $R_2= 0$. 
Since $R_2=0$, we have $K_{-2}=0$.  
By $(**)$ we have $R_3 \ne 0$ and 
$\dim_k R_3 = \dim_k K_{-3}+1$. 
We must have $R_3 = K_{-3}L_6$ since $K_{-2}=0$. 
Hence we must have $K_{-3}\ne 0$,  
$\dim_k R_3=2$ and $\dim_k K_{-3}=1$, since $\dim_k R_3=3$ is impossible. 
Also, since 
$L_5\cdot K_{-3}
\subset R_2 =0$, we must have $L_5=0$. 
\par
Since $R_3 = K_{-3}L_6$, this implies that $\dim_k L_6=2$. 
Since 
$\deg (-K_C + 6D)=2$, 
this is only possible if $C$ is hyperelliptic and $-K_C + 6D \sim 2P$ for some hyperelliptic 
point $P\in C$. 
\par
Next, since $4D \not\sim K_C$ (otherwise $R$ is Gorenstein), we have $\dim_k R_4=2$
 and we must have $R_4 = K_{-3}L_7$ since $K_{-2}=L_5=0$. This implies that
  $\dim_k L_7=2$. 
Now, $-K_C +7D \sim 2P +D$ and $h^0(2P+D) - h^1(2P+D)=1$, we have 
$h^0(2P -D)=h^1(2P+D)= h^0(K_C -(2P+D))= 1$, since we have $K_C \sim 4P$. 
Hence we have $2P \sim D + Q$ for some point $Q\in C$. But then we will have            
$h^0(D) = h^0 (2P - Q) \ge 1$, contradicting our assumption $h^0(D)=0$.

\par \vspace{1mm} 
(b) Next assume that $R_2 \ne 0$. Since $a(R)\le 3$ and $b(R)\ge 5$, we must have 
$R_2 = K_{-3}\cdot L_5$. Since $\deg (K_C -3D) 
=\deg (-K_C + 5D) = 1$, we must have 
$\dim_k R_2=\dim_k K_{-3}=\dim_k L_5=1$ and we have also 
$\dim_k R_3 = 2$, $K_C \sim 3D+P, 5D \sim K_C + Q$ 
hence $2D \sim P+Q$ for some points 
$P, Q \in C$.  
Now we must have $R_3 = K_{-2}L_5 + K_{-3}L_6$. 
Since $5D\sim K_C+Q$, we get
\[
L_5 =H^0(C,\cO_C(Q)), \quad  
L_6= H^0(C,\cO_C(D+Q)).
\] 
Also, 
\[
K_{-2}=H^0(C,\cO_C(D+P)), \quad K_{-3}=H^0(C,\cO_C(P)). 
\]

\par
Next we look at $L_6$. 
Since $6D \sim 2K_C - 2P$, 
we have $L_6 = H^0(C,\cO_C(K_C - 2P))$ and 
$\dim_k L_6 \ge \dim_k H^0(C,\cO_C(K_C)) -2 =1$.  
If $\dim_k L_6 = 1$, then we have $\dim_k H^0(C,\cO_C(2P)) =2$ and $C$ is hyperelliptic and 
$P$ is a hyperelliptic point. 
Then $4D\sim K_C$ and $R$ is Gorenstein. 
If $\dim_k L_6=2$, 
since $\deg (-K_C+6D) = 2$, 
then again $K_C- 2P\sim 2P'$ for some hyperelliptic point $P'$ and 
then again $4D\sim K_C$ and $R$ is Gorenstein.

\par \vspace{1mm}
If $\deg D=2$, then $H^0(C,\cO_C(D))=0$ by Proposition 
\ref{degD|2g-2}. 
Then by $(**)$, we have $H^0(C, \cO_C(K_C -D))=0$; namely, $K_{-1} =0$. 
Since $\deg (2D) =4$, we have $R_2 \ne 0$. But since $L_2=K_{-1}=0$, 
$\KiR\cdot K_R \subset \bigoplus_{n\ge 3}R_n$ and $R$ is not nearly Gorenstein. 
\par \vspace{1mm}
(2) See Proposition \ref{d=g-1} (3). 
 
\par \vspace{1mm}
(3) It follows from Proposition \ref{degD|2g-2}. 

\par \vspace{2mm}
 Now suppose $\deg D \ge 5$. 
Then we must have $a(R)=0, \dim_k K_0=3$ and 
$\dim_k R_1 = \deg D -2,  
m(R)=1$, $b(R)=1$ and $R_1= K_0\cdot L_1$.

\par \vspace{1mm}
(4) This follows from Proposition \ref{d=g-1} (4). 

\par \vspace{1mm}
(5) If $\deg D =6$, then $\dim_k R_1 = 4$. 
 If $R$ is nearly Gorenstein, then  
we muct have $R_1 = K_0\cdot L_1$.  Since $\dim_k K_0=3$, 
we must have $\dim_k L_1>1$. 
Since $\deg (- K_C +D) =2$, $C$ must be hyperelliptic and 
$D-K_C \sim 2P$, where $P$ is a hyperelliptic point of $C$. Since $K_C \sim 4P$ then, 
we must have $D\sim 6P$. Conversely, if $P$ is a hyperelliptic point and $D=6P$, 
then $R$ is nearly Gorenstein by Proposition \ref{hypell} (4)(b).
\par
\vspace{1mm}
If $\deg(D) \ge 7$,
then $R$ is a standard graded ring by \cite{Mum} and 
by Proposition \ref{a(R)b(R)},   
$R$ is nearly Gorenstein if and only if $R_1 = K_0\cdot L_1$.
This explains (8). For $d=7, 8$, we have more explicit characterization 
for $R$ to be nearly Gorenstein.  

\par \vspace{1mm}
(6) 
Assume $\deg D =7$ and that $R$ is nearly Gorenstein. 
Then we must have  $R_1 = K_0\cdot L_1$. 
Since  $\dim_k R_1 = h^0(D)= 5$ and  $\dim_k K_0=3$, 
we must have $\dim_k L_1=h^0(D-K_C)>1$. 
Hence $D-K_C=B$ is an effective divisor of degree $3$ and 
$h^0(B)\ge 2$. 
By ($*$), we have $h^0(B) = 2$ and $h^0(K_C - B)=1$. 
Then we have $K_C = B+Q$ for some point $Q\in C$.  
\par \vspace{1mm}
Now, we will show that $R_1= K_0L_1$ if $\cO_C(B)$ has no base point.  
We have seen that $h^0(B)= h^0( K_C - Q)=2$. 
Let $(a,b)$ be basis of \par\noindent
$H^0(C,\cO_C(B))$. 
First assume that $\cO_C(K_C-Q)$ has no base point. 
Then by multiplication of $(a, b)$, 
we get an exact sequence 
\[
0 \to \cO_C(-B)\to \cO_C^{\oplus 2} \to \cO_C(B)\to 0.
\]
Tensoring $\cO_C(K_C)$, noting $B = K_C-Q$ and $D=K_C+B$, we have      
\[
0 \to \cO_C(Q) \to \cO_C(K_C)^{\oplus 2} \to \cO_C(2K_C - Q)= \cO_C(D) \to 0.
\] 
Since  $h^1(Q) = 2 =  2 \cdot h^1(K_C)$ and $h^1(D)=0$, 
the mapping $H^0(\cO_C(K_C)^{\oplus 2}) \to H^0(\cO_C(B))$ is surjective and hence 
we have $R_1 = K_0 \cdot L_1$. 
\par 
If $\cO_C(B)$ has a base point $P$, then $L_1 = H^0(C,\cO_C(B)) = H^0(C,\cO_C(B-P))$. 
Hence $K_0\cdot L_1 \subset H^0(C,\cO_C(K_C + B - P))$. But since 
$H^0(C,\cO_C(K_C + B - P))\subsetneq  H^0(C,\cO_C(K_C + B))=R_1$, $R$ is not nearly Gorenstein. 

\par \vspace{1mm}
(7) Let $\deg D=8$ and assume $R_8 = K_0\cdot L_8$. If $8D - K_C \not\sim K_C$, 
then $\dim_k L_8=2$. Since $\dim_k L_8=2$ and $\dim_k K_0 =3$, we can easily see 
$\dim_k K_0\cdot L_8 < 6$ and $R$ is not nearly Gorenstein. Conversely, if $R\sim 2K_C$, 
and $C$ is not hyperelliptic, then $R$ is isomorphic to the 2nd Veronese subring of 
$R(C,K_C)$ and is nearly Gorenstein by Proposition \ref{R(d)stand}. 
If $C$ is hyperelliptic, 
then since $g=3$ and $d=8$, $R$ is not nearly Gorenstein 
by Proposition \ref{hypell} (4)(b).   
\par \vspace{1mm} 
(8) If $\deg D>12$, then $R$ is nearly Gorenstein by Theorem 
\ref{high-deg}. 
If $\deg D= 12$ and $H^1(C,\cO_C(D-2K_C))= 0$, then $R$ is nearly Gorenstein 
by Theorem \ref{high-deg} (1).  
If $\deg D= 12$ and $H^1(C,\cO_C(D-2K_C))\ne 0$, then 
$D-2K_C\sim K_C$ and 
$D\sim 3K_C$. 
If $C$ is hyperelliptic, then $R$ is $12$-th Veronese subring as in 
Theorem \ref{hypell}(4)(b) and is nearly Gorenstein. 
If $C$ is not hyperelliptic, 
then $R\cong R(C,K_C)^{(3)}$ is nearly Gorenstein 
by Proposition \ref{R(d)stand}.      
\end{proof} 

\begin{ex} If $g=3, \deg D=9$ and $D = 2K_C + P$ for some point 
$P\in C$, 
then, since $H^0(C, \cO_C(K_C+P))= H^0(C,\cO_C(K_C))$, 
$K_0\cdot L_1 \subset H^0(C,\cO_C(2K_C))
\subsetneq H^0(C,\cO_C(D))=R_1$ and $R$ is not nearly Gorenstein. 
Is this the only example that $R(C,D)$ is not  nearly Gorenstein when $g=3$ and $\deg D \ge 9$?    
\end{ex}

\medskip
\section{Comparison of nearly Gorenstein and almost 
Gorenstein property for  $2$-dimensional cone singularities}
\label{s:almost}

We are also interested in comparison of {\em almost Gorenstein} rings and nearly Gorenstein 
rings. 
The notion of almost Gorenstein local rings was introduced by Barucci and Fr\"oberg \cite{BF} for one-dimensional local rings
and  has been generalized by Goto, Takahashi, and Taniguchi \cite{GTT} for Cohen-Macaulay local rings having canonical modules of any dimension.

\begin{noname}\label{AGdef}(\cite{GTT}) 
Let $(A,\m)$ be a Cohen-Macaulay local ring of dimension $d$
and $\omega  \in K_A$ be a general element of the canonical module $K_A$.  
We put  $U = K_A/ \omega A$. \par
We say that $A$ is {\em almost Gorenstein} if $U$  satisfies the equality
\[
\mu(U) = e_0(U),
\]   
where  $\mu(U) = \dim_{A/\m}( U /\m U)$ is the number of minimal generators of $U$ and 
$e_0(U)$ denotes the multiplicity of $U$ as an $A$-module of dimension $d-1$. (An $A$-module which satisfies this property is called an  {\em Ulrich $A$-module}.)
\end{noname}

\begin{rem}[\textrm{cf. \cite[Proposition 6.1]{HHS}}] 
If $A$ is an almost Gorenstein ring of $\dim A=1$, then it is also nearly Gorenstein. 
\end{rem}

\par 
So it is natural to hope that \lq\lq almost Gorenstein" property is \lq\lq stronger" than 
\lq\lq nearly Gorenstein" property also in dimension $2$. But we will see in the following, 
this guess is completely false.    

\begin{rem} In dimension $2$, it is proved in \cite{GTT} that rational singularities are 
almost Gorenstein. Also, in \cite{OWYaG}, the authors proved that 
\lq\lq elliptic singularities"
(\ref{pa(Z)}) are almost Gorenstein.  
See also Propositions \ref{R(C,D)rational} and \ref{ell1}. 
\end{rem}

\par 
The following statement is Theorem 6.21 of \cite{OWYaG}. 
We can see that the results are 
completely different 
from our results in Theorem \ref{g=2}.

\begin{thm}[\textrm{\cite[Theorem 6.21]{OWYaG}}] 
Let $R=R(C,D)$ with $g(C)=2$. Then we have the following;
\begin{enumerate}  
\item If $\deg(D) \le 2$, then $R$ is almost Gorenstein.
\item If $\deg(D) = 3$, $R(C,D)$ is almost Gorenstein if and only if $\cO_C(D)$ is 
{\em not} generated.
\item If $\deg(D) \ge 4$ and $D \not\sim 2K_C$, then $R$ is {\em not} 
almost Gorenstein. 
If $\deg(D) \ge 4$ and $D\sim 2K_C$, then $R$ is almost Gorenstein. 
\end{enumerate}
\end{thm}

\par 
Now we compare this result and Theorem \ref{g=2}, 
we get the following result 
showing \lq\lq almost Gorenstein" and \lq\lq nearly Gorenstein" 
properties are
almost contradictory for cone singularities of $g=2$. 

\begin{thm} Let $R=R(C,D)$, $D$ is an  {\em integral} ample 
divisor on a smooth curve $C$ of genus $2$. 
\begin{enumerate}
\item[$(1)$] If $R$ is nearly Gorenstein, almost Gorenstein and not Gorenstein, then 
we have either;
\begin{enumerate} 
\item[\rm{(a)}]  $\deg D = 3$ and $D\sim K_C +P$ for some point $P\in C$.  
\item[\rm{(b)}]  $D\sim 2K_C$.  
\end{enumerate}
\item[$(2)$] 
If $R$ is not Gorenstein, not almost Gorenstein and not nearly Gorenstein, 
then $D$ is one of the following$:$
\begin{enumerate}
\item[\rm{(a)}]  $\deg D =3$ and $\cO_C(D)$ is generated by global sections.
\item[\rm{(b)}] $\deg D =5$ and $D \sim 3K_C -P$ for some point $P\in C$. 
\item[\rm{(c)}]  $\deg D=6$ and $D\sim 3K_C$. 
\end{enumerate}
\end{enumerate}
\end{thm} 

\begin{acknowledgement}
The authors thank the referee for careful reading of our manuscript. 
\end{acknowledgement}


\end{document}